\documentclass[a4paper,oneside,10pt]{article} 

\title{ARTICLE - THE CRAMER CONDITION FOR THE CURIE-WEISS MODEL OF SOC}
\author{Matthias Gorny}
\date{\today}

\usepackage[english]{babel}
\usepackage[latin1]{inputenc}
\usepackage[T1]{fontenc}
\usepackage{amsfonts}
\usepackage{graphicx}
\usepackage[cyr]{aeguill}
\usepackage{xspace}
\usepackage{tikz}
\usetikzlibrary{patterns}
\usepackage{cite}
\usepackage{tocloft}
\usepackage{amsthm}

\usepackage{amssymb}
\usepackage{amsmath}
\usepackage{stmaryrd}
\usepackage{latexsym}
\usepackage{exscale}
\usepackage{relsize}
\usepackage{accents}
\usepackage{dsfont}
\usepackage{mathrsfs}
\usepackage{yhmath}

\newcommand{\ind}[1]{\mathds{1}_{#1}}
\newcommand{\N}{\mathbb{N}}
\newcommand{\Z}{\mathbb{Z}}

\newcommand{\R}{\mathbb{R}}
\newcommand{\C}{\mathbb{C}}

\renewcommand{\P}{\mathbb{P}}
\newcommand{\E}{\mathbb{E}}

\newcommand{\Bl}{\mathrm{B}}

\newcommand{\Cc}{\mathcal{C}}

\newcommand{\Fc}{\mathcal{F}}

\newcommand{\Lc}{\mathcal{L}}

\newcommand{\Uc}{\mathcal{U}}
\newcommand{\Vc}{\mathcal{V}}

\newcommand{\Xc}{\mathcal{X}}

\newcommand{\Ck}[1]{\mathcal{C}^{#1}}

\def\Dro{\smash{{D}^{\!\!\!\!\raise4pt\hbox{$\scriptstyle o$}}}}
\def\Ccro{\smash{{\mathcal{C}}^{\!\!\!\raise4pt\hbox{$\scriptstyle o$}}}}
\def\Aro{\smash{{A}^{\!\!\!\raise5pt\hbox{$\scriptstyle o$}}}}
\def\Bro{\smash{{B}^{\!\!\!\raise5pt\hbox{$\scriptstyle o$}}}}

\newcommand{\gfrac}[2]{\genfrac{}{}{0pt}{1}{#1}{#2}}

\newcommand{\limsupn}{\underset{n \to +\infty}{\mathrm{limsup}}\,\,}
\newcommand{\liminfn}{\underset{n \to +\infty}{\mathrm{liminf}}\,\,}
\newcommand{\limsupk}{\underset{k \to +\infty}{\mathrm{limsup}}\,\,}

\renewcommand{\a}{\alpha}
\renewcommand{\b}{\beta}
\newcommand{\g}{\gamma}
\newcommand{\G}{\Gamma}
\renewcommand{\d}{\delta}
\newcommand{\D}{\Delta}
\renewcommand{\epsilon}{\varepsilon}
\newcommand{\eps}{\varepsilon}
\newcommand{\z}{\zeta}
\renewcommand{\r}{\rho}
\newcommand{\s}{\sigma}

\renewcommand{\phi}{\varphi}

\renewcommand{\l}{\lambda}
\renewcommand{\L}{\Lambda}
\newcommand{\x}{\xi}

\renewcommand{\o}{\omega}
\renewcommand{\O}{\Omega}

\newtheorem{theo}{Theorem}

\newtheorem{prop}[theo]{Proposition}

\newtheorem{lem}[theo]{Lemma}

\renewenvironment{proof}{\noindent{\bf Proof.}}{\qed}

\begin{document}

\renewcommand{\contentsname}{Contents}
\renewcommand{\refname}{\textbf{References}}
\renewcommand{\abstractname}{Abstract}

\begin{center}
\begin{Huge}
The Cram\'er Condition for\medskip

the Curie-Weiss Model of SOC
\end{Huge} \bigskip\bigskip \bigskip \bigskip

\begin{Large} Matthias Gorny \end{Large} \smallskip
 
\begin{large} {\it Universit\'e Paris Sud \emph{and} ENS Paris} \end{large} \bigskip \bigskip

\end{center}
\bigskip \bigskip \bigskip

\begin{abstract}
\noindent We pursue the study of the Curie-Weiss model of self-organized criticality we designed in~\cite{CerfGorny}. We extend our results to more general interaction functions and we prove that, for a class of symmetric distributions satisfying a Cram\'er condition $(C)$ and some integrability hypothesis, the sum $S_{n}$ of the random variables behaves as in the typical critical generalized Ising Curie-Weiss model. The fluctuations are of order $n^{3/4}$ and the limiting law is $k \exp(-\lambda x^{4})\,dx$ where $k$ and $\lambda$ are suitable positive constants. In~\cite{CerfGorny} we obtained these results only for distributions having an even density.
\end{abstract}
\bigskip \bigskip \bigskip \bigskip \bigskip



\newpage

\section{Introduction}
\label{Intro}

\noindent In~\cite{CerfGorny}, we introduced a Curie-Weiss model of self-organized criticality (SOC): we transformed the distribution associated to the generalized Ising Curie-Weiss model by implementing an automatic control of the inverse temperature which forces the model to evolve towards a critical state.\medskip

\noindent We proved rigorously that this model exhibits a phenomenon of self-organized criticality: if we build the model with a probability $\r$ having an even density which satisfies some integrability conditions, then, asymptotically, the sum $S_n$ of the random variables behaves as in the typical critical generalized Ising Curie-Weiss model. The fluctuations of $S_{n}$ are of order $n^{3/4}$ and the limiting law~is
\[\left(\frac{4}{3}\right)^{1/4}\G\left(\frac{1}{4}\right)^{-1} \exp\left(-\frac{s^{4}}{12}\right)\,ds\,.\]
Our result presents an unexpected universal feature. Indeed, this is in contrast to the situation in the critical generalized Ising Curie-Weiss model: at the critical point, the fluctuations are of order $n^{1-1/2k}$, where $k$ depends on the distribution~$\r$. Moreover our integrability conditions on $\r$ are weaker than those required to define the generalized Ising Curie-Weiss model, studied by Richard S. Ellis and Charles M. Newman in~\cite{EN}. For instance, our result holds for any centered Gaussian measure on $\R$.\medskip

\noindent The hypothesis that the law $\r$ has a density is essential in the proof of the fluctuations result in~\cite{CerfGorny}. Here we use arguments coming from the work of Anders Martin-L\"of~\cite{MartinLof} to extend this result to any symmetric probability measure which satisfies some integrability hypothesis and a Cram\'er condition:
\begin{equation}
\forall \a>0 \qquad \sup_{\|(s,t)\|\geq \a}\left|\int_{\R}e^{isz+itz^2}\,d\r(z)\right|<1.\tag{$C$}
\end{equation}
This includes a much larger class of probability measures. However the proof is much more technical. We also solve the problem of the mass at $0$ of $\r$ that we met in~\cite{CerfGorny} and we extend the law of large numbers associated to our model.\medskip

\noindent In this paper, we also extend our results to more general interaction functions. This extension is similar in spirit to the work of Richard S. Ellis and Theodor Eisele~\cite{EE} in the context of the generalized Ising Curie-Weiss model.\medskip

\noindent \textbf{The model.} Let $g$ be a measurable real-valued function defined on $\R$ such that $g(u)\sim u^2/2$ in the neighbourhood of~$0$ and
\[\forall u\in \R\qquad g(u)\leq \frac{u^2}{2}\,.\]
Let $\r$ be a probability measure on $\R$, which is not the Dirac mass at 0. We consider an infinite triangular array of real-valued random variables $(X_{n}^{k})_{1\leq k \leq n}$ such that, for all $n \geq 1$, $(X^{1}_{n},\dots,X^{n}_{n})$ has the distribution $\widetilde{\mu}_{n,\r,g}$, whose density with respect to $\r^{\otimes n}$ is
\[(x_{1},\dots,x_{n})\longmapsto\frac{1}{Z_{n,g}}\exp\left(
ng\left(\frac{x_1+\dots+x_n}{\sqrt{n(x_1^2+\dots+x_n^2)}}\right)
\right)\ind{\{x_{1}^{2}+\dots+x_{n}^{2}>0\}}\,,\]
where
\[Z_{n,g}=\int_{\R^{n}}\exp\left(
ng\left(\frac{x_1+\dots+x_n}{\sqrt{n(x_1^2+\dots+x_n^2)}}\right)
\right)\ind{\{x_{1}^{2}+\dots+x_{n}^{2}>0\}}\,\prod_{i=1}^{n}d\r(x_{i})\,.\]
We define $S_{n}=X^{1}_{n}+\dots+X^{n}_{n}$ and $T_{n}=(X^{1}_{n})^{2}+\dots+(X^{n}_{n})^{2}$.\medskip

\noindent We state next our main result, which is a strengthening of theorems $1$ and $2$ of~\cite{CerfGorny}:

\begin{theo} Let $\r$ be a symmetric probability measure on $\R$ with positive variance $\s^{2}$ and such that
\[\exists v_{0}>0 \qquad \int_{\R}e^{v_{0}z^{2}}\,d\r(z)<+\infty\,.\]

\noindent \textup{Law of large numbers:} Under $\widetilde{\mu}_{n,\r,g}$, $(S_{n}/n,T_{n}/n)$ converges in probability towards $(0,\s^{2})$.\medskip

\noindent We suppose in addition that $g$ has a fourth derivative at~$0$ and that the following Cram\'er condition holds: 
\begin{equation}
\forall \a>0 \qquad \sup_{\|(s,t)\|\geq \a}\left|\int_{\R}e^{isz+itz^2}\,d\r(z)\right|<1.
\label{(C)} \tag{$C$}
\end{equation}
Let $\mu_{4}$ be the fourth moment of $\r$. We denote $m_4=-g^{(4)}(0)/2\geq 0$.\medskip

\noindent \textup{Fluctuations result:} Under $\widetilde{\mu}_{n,\r,g}$,
\[\left(\mu_{4}+m_4\s^4\right)^{1/4}\,\frac{S_{n}}{\s^{2}n^{3/4}} \overset{\Lc}{\underset{n \to \infty}{\longrightarrow}} \left(\frac{4}{3}\right)^{1/4}\G\left(\frac{1}{4}\right)^{-1} \exp\left(-\frac{s^{4}}{12}\right)\,ds\,.\]
\label{MainTheorem}
\end{theo}

\noindent The condition $(C)$ is called the Cram\'er condition for the law of $(Z,Z^2)$, where~$Z$ is a random variable with distribution $\r$. The class of probability measures satisfying $(C)$ is much larger than the class of probability measures having a density. Indeed, by the Lebesgue decomposition theorem (see~\cite{Rudin}), there exist three non-negative real numbers $a,b,c$ such that $a+b+c=1$ and
\[\r=a\,\r_{ac}+b\,\r_{d}+c\,\r_{s}\,,\]
where $\r_{ac}$ is a probability measure with density $f$, $\r_{d}$ is a discrete probability measure and $\r_{s}$ is a singular probability measure having no atoms. If $a>0$, we say that $\r$ has an absolutely continuous component.

\begin{prop} If $\r$ has an absolutely continuous component, then
\[\forall \a>0 \qquad \sup_{\|(s,t)\|\geq \a}\left|\int_{\R}e^{isz+itz^2}\,d\r(z)\right|<1\,.\]
\label{a>0Cramer}
\end{prop}

\noindent For example, the law
\[\r_0=\frac{1}{16}\d_{-1}+\frac{3}{4}\d_{0}+\frac{1}{16}\d_{1}+\exp\left(-\frac{x^2}{2}\right)\,\frac{dx}{8\sqrt{2\pi}}\]
satisfies the hypothesis of theorem~\ref{MainTheorem}.\medskip

\noindent In~\cite{CerfGorny}, we treated the case where $g(u)=u^2/2$ for any $u\in \R$. We obtained a law of large numbers under $\widetilde{\mu}_{n,\r,g}$, for symmetric probability measures $\r$ such that $\r(\{0\})<e^{-1/2}$ or such that $\r(]0,c[)=0$ for some $c>0$. The above distribution $\r_0$ does not satisfy this hypothesis. Moreover, in the fluctuations theorem of~\cite{CerfGorny}, we only deal with a distribution $\r$ having an even density $f$ which satisfies
\[\int_{\R^{2}}f^{p}(x+y)f^{p}(y)|x|^{1-p}\,dx\,dy<+\infty\,,\]
for some $p \in \, ]1,2]$: once again this is not the case for $\r_0$. Hence theorem~\ref{MainTheorem} improves the main results of~\cite{CerfGorny}. Yet its proof is much more complicated: we have to use an approximation of the identity to obtain an asymptotic relation between $\nu_{\r}^{*n}$ and its Cram\'er transform. The final Laplace's method is also much more technical than in~\cite{CerfGorny}. \medskip

\noindent Remark: If we start with the model studied in~\cite{EE} and we follow the same road as in~\cite{CerfGorny}, then we end up with the distribution $\widetilde{\mu}^{\star}_{n,\r,g}$ whose density with respect to $\r^{\otimes n}$ is
\[(x_{1},\dots,x_{n})\longmapsto\frac{1}{Z^{\star}_{n,g}}\exp\left(
n^2\,\frac{g\big((x_1+\dots+x_n)/n\big)}{x_{1}^{2}+\dots+x_{n}^{2}}\right)\ind{\{x_{1}^{2}+\dots+x_{n}^{2}>0\}}\,,\]
where $Z^{\star}_{n,g}$ is the renormalization constant. In this case, the result stated in theorem~\ref{MainTheorem} holds as well, but with $\smash{(\mu_4+m_4\s^6)^{1/4}}$ instead of $\smash{(\mu_4+m_4\s^4)^{1/4}}$.\medskip

\noindent Before we do the proof of theorem~\ref{MainTheorem} in section~\ref{ProofMainTheo}, we give some preliminaries in section~\ref{Preliminary} and we extend the results of~\cite{CerfGorny} around Varadhan's lemma in section~\ref{AroudVaradhan}. Next, in section~\ref{sectionCramer}, we give some generalities on the Cram\'er condition, we prove proposition~\ref{a>0Cramer} and we show an asymptotic relation with the Cram\'er transform.

\section{Preliminaries}
\label{Preliminary}

\noindent Here we give some notations and results derived from the sections $3$ and $5$ of~\cite{CerfGorny} and which are essential for the proof of theorem~\ref{MainTheorem}.\medskip

\noindent Let $F$ and $F_{\!g}$ be the functions defined on $\R\times\,]0,+\infty[$ by
\[\forall (x,y) \in \R\times\,]0,+\infty[ \qquad F(x,y)=\frac{x^2}{2y}\qquad\mbox{and}\qquad F_{\!g}(x,y)= g\left(\frac{x}{\sqrt{y}}\right)\,.\]
We define the sets 
\[\D=\{\,(x,y) \in \R^2 : x^2\leq y\,\} \qquad \mbox{and} \qquad \D^{\!*}=\D\backslash\{(0,0)\}\,.\]
We denote by $\nu_{\r}$ the law of $(Z,Z^2)$, where $Z$ is a random variable with distribution $\r$, and by $\widetilde{\nu}_{n,\r}$ the law of $(S_n/n,T_n/n)$ under $\r^{*n}$. Under $\widetilde{\mu}_{n,\r,g}$, the law of $(S_n/n,T_n/n)$ is
\[\frac{\displaystyle{\exp\left(nF_{\!g}(x,y)\right)\ind{\D^{\!*}}(x,y)\,d\widetilde{\nu}_{n,\r}(x,y)}}{\displaystyle{\int_{\D^{\!*}}\exp\left(nF_{\!g}(s,t)\right)\,d\widetilde{\nu}_{n,\r}(s,t)}}\,.\]

\noindent Let $\r$ be a symmetric probability measure on $\R$ with variance $\s^2$. We define the Laplace transform $\L$ of $\nu_{\r}$ by
\[\forall(u,v) \in \R^2 \qquad \L(u,v)=\ln \int_{\R}e^{uz+vz^2}\,d\r(z)\]
and by $D_{\L}$ the set of the points $(u,v) \in \R^2$ such that $\L(u,v)<+\infty$. We define next the Cram\'er transform $I$ of $\nu_{\r}$ by
\[\forall (x,y) \in \R^2 \qquad I(x,y)=\sup_{(u,v) \in \R^2} \,(\,ux+vy-\L(u,v)\,)\]
and by $D_{I}$ the set of the points $(x,y) \in \R^2$ such that $I(x,y)<+\infty$.\medskip

\noindent We suppose that $(0,0) \in \Dro_{\L}$. Then $I$ is a good rate function, i.e., it is non-negative and for any $\a>0$, the set $\{\,(x,y) \in \R^2 : I(x,y) \leq \a\,\}$ is compact. Moreover Cram\'er's theorem states that $(\widetilde{\nu}_{n,\r})_{n \geq 1}$ satisfies a large deviations principle, with speed $n$, governed by $I$.\medskip

\noindent Next $I(0,0)=-\ln \r(\{0\})$ and the function $I-F$ has a unique minimum on~$\D^{\!*}$ at $(0,\s^2)$, with $(I-F)(0,\s^2)=0$. Moreover, if the support of $\r$ contains at least three points and if $\mu_4$ denotes the fourth moment of $\r$, then, when $(x,y)$ goes to $(0,\s^2)$,
\[I(x,y)-F(x,y) \sim \frac{\mu_4 x^4}{12\s^8}+\frac{(y-\s^2)^2}{2(\mu_4-\s^4)}\,.\]
Finally, since $g$ has a fourth derivative at~$0$, the Taylor-Young formula implies that
\[g(u)=\frac{u^2}{2}+g^{(3)}(0)\frac{u^3}{6}-m_4\frac{u^4}{12}+o(u^4)\,.\]
We have $g(u)\leq u^2/2$ for any $u\in \R$. Therefore $g^{(3)}(0)=0$, $m_4\geq 0$ and thus, when $(x,y)$ goes to $(0,\s^2)$,
\[F(x,y)-F_{\!g}(x,y) =\frac{m_4 x^4}{12y^2}(1+o(1))=\frac{m_4 x^4}{12\s^4}+o(\|(x,y)\|^4)\,.\]
As a consequence
\[I(x,y)-F_{g}(x,y)\sim \frac{(\mu_4+m_4\s^4) x^4}{12\s^8}+\frac{(y-\s^2)^2}{2(\mu_4-\s^4)}\,.\]\medskip

\noindent Remark: In the case of the model given by the distribution $\widetilde{\mu}^{\star}_{n,\r,g}$, defined in the remark at the end of the introduction, we replace $F_{\!g}$ by the function $(x,y)\in\R\times\,]0,+\infty[\,\longmapsto g(x)/y$ in the sections 2-5. The only difference is that, when $(x,y)$ goes to $(0,\s^2)$,
\[I(x,y)-F_{g}(x,y) \sim\frac{(\mu_4+m_4\s^6) x^4}{12\s^8}+\frac{(y-\s^2)^2}{2(\mu_4-\s^4)}\,.\]

\newpage

\section{Around Varadhan's lemma}
\label{AroudVaradhan}

\noindent In section $6$ of~\cite{CerfGorny}, we proved the following result:

\begin{lem} Let $\r$ be a symmetric probability measure on $\R$ such that $(0,0) \in \Dro_{\L}$ and $\r(\{0\})=0$. Let $\s^2$ denote the variance of $\r$. If $A$ is a closed subset of $\R^{2}$ which does not contain $(0,\s^{2})$ then  
\[\limsupn \frac{1}{n}\ln \int_{\D^{\!*}\cap A}\exp\left(\frac{nx^{2}}{2y}\right)\,d\widetilde{\nu}_{n,\r}(x,y)<0\,.\]
\label{TypeVaradhan}
\end{lem}

\noindent Actually we obtained in~\cite{CerfGorny} this same conclusion for symmetric measures $\r$ such that $\r(\{0\})<e^{-1/2}$ or such that $\r(]0,c[)=0$ for some $c>0$. This restriction is due to the behaviour of $I-F$ near the point $(0,0)$, which is a singularity of $F$.\medskip

\noindent In this section, we will extend this result to any non-degenerate symmetric probability measure on $\R$ such that $(0,0) \in \Dro_{\L}$. To this end, we will rely on a conditioning argument in order to reduce the problem to the case of measures which have no point mass at $0$, and to apply lemma~\ref{TypeVaradhan}. We focus first on what happens in the neighbourhood of $(0,0)$.

\begin{prop} Suppose that $\r$ is a symmetric probability measure on $\R$ with positive variance $\s^{2}$ and such that $(0,0) \in \Dro_{\L}$. There exists $\g>0$ such that, for $\d \in \,]0,\s^2[$ small enough and for $n$ large enough,
\[\int_{\D^{\!*}}e^{nx^2/(2y)}\ind{0<y\leq \d}\,d\widetilde{\nu}_{n,\r}(x,y) \leq e^{-n\g}\,.\]
\label{TypeVaradhan(0,0)}
\end{prop}

\noindent We notice that the constant $\g$ only depends on $\r$ (and not $\d$).\medskip

\begin{proof} If $\r(\{0\})=0$ then lemma~\ref{TypeVaradhan} implies that the constant
\[\g=-\frac{1}{2}\,\limsupn \frac{1}{n}\ln \int_{\D^{\!*}} e^{nx^2/(2y)}\ind{0<y\leq \s^2/2}\,d\widetilde{\nu}_{n,\r}(x,y)\]
is positive since $\{\,(x,y)\in \R^2:0\leq y \leq \s^2/2\,\}$ is a closed set which does not contain $(0,\s^{2})$. For $\d \in \,]0,\s^2/2[$, we have then 
\[ \limsupn \frac{1}{n}\ln \int_{\D^{\!*}} e^{nx^2/(2y)}\ind{0<y\leq \d}\,d\widetilde{\nu}_{n,\r}(x,y) \leq-2\g <-\g\,.\]
Hence the result holds for probability measures which have no point mass at~$0$.\medskip

\noindent We suppose now that $\r(\{0\})>0$. Let $n\geq 1$ and $X_1,\dots,X_n$ be independent random variables with common distribution $\r$. We put
\[ S_n=\sum_{i=1}^{n}X_i \qquad \mbox{and} \qquad T_n=\sum_{i=1}^{n}X_i^2\,.\]
For $\d>0$ small enough, we denote
\[ E_{n,\d}=\int_{\D^{\!*}}e^{nx^2/(2y)}\ind{0<y\leq \d}\,d\widetilde{\nu}_{n,\r}(x,y)\,.\]
Since $\widetilde{\nu}_{n,\r}(\D)=1$, we have
\[E_{n,\d}=\E\left(e^{S_n^2/(2T_n)}\ind{0<T_n\leq n\d}\right)\,.\]
For any $c>0$, we have
\[ E_{n,\d}\leq\E\left(e^{S_n^2/(2T_n)}\ind{T_n>0}\ind{T_n/n\leq c|S_n/n|}\right) +\E\left(e^{S_n^2/(2T_n)}\ind{c|S_n/n|<T_n/n\leq \d}\right)\]
and we write this sum $I_{n,1}+I_{n,2}$.\medskip

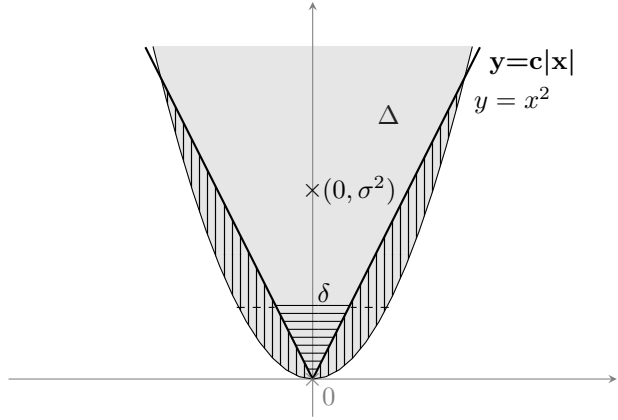
\begin{center}
\begin{tikzpicture}
\filldraw [color=gray!20] (2.1,4.41) -- (-2.1,4.41) -- plot [domain=-2.1:2.1] (\x,{(\x)*(\x)}) -- cycle ;
\draw [>=stealth,->] [color=gray!100] (-4,0) -- (4,0) ;
\draw [>=stealth,->] [color=gray!100] (0,-0.5) -- (0,5) ;
\draw [domain=-2.1:2.1] plot (\x,{(\x)*(\x)}) ;
\draw [thick] [domain=0:2.2] plot (\x,2*\x) ;
\draw [thick] [domain=-2.2:0] plot (\x,-2*\x) ;
\draw [color=gray!100](0,0) node[below right] {$0$} ;
\draw [color=gray!100](0,0) node {$\times$} ;
\draw (2,4) node[below right] {$y=x^{2}$} ;
\draw (2.2,4.5) node[below right] {$\textbf{y=c|x|}$} ;
\draw [color=gray!175] (1,3.5) node {$\D$} ;
\draw (0.15,1.15) node {$\d$} ;
\draw (0,2.5) node {$\times$} ;
\draw (0,2.5) node[right] {$(0,\s^2)$} ;
\fill [pattern=vertical lines] (0,0) -- plot [domain=-2:0](\x,{(\x)*(\x)}) --(-2,4) -- cycle;
\fill [pattern=vertical lines] (0,0) -- plot [domain=0:2](\x,{(\x)*(\x)}) --(2,4) -- cycle;
\fill [pattern=horizontal lines] (0,0) --(1/2,1) -- (-1/2,1) -- cycle;
\draw [dashed] (-1,0.95) -- (-1/2,0.95) ;
\draw [dashed] (1,0.95) -- (1/2,0.95) ;
\end{tikzpicture}
\end{center}

\noindent In the figure, $I_{n,1}$ is an integral on the vertically hatched area and $I_{n,2}$ is an integral on the horizontally hatched area.
\medskip

\noindent We notice that, if $c|S_n/n|<T_n/n\leq \d$, then
\[ \frac{S_n^2}{2T_n}\leq \frac{T_n^2}{2c^2T_n}\leq \frac{T_n}{2c^2}\leq \frac{n\d}{2c^2}\,.\]
We have thus
\[ I_{n,2}\leq \exp\left(\frac{n\d}{2c^2}\right)\,\P\left(c\left|\frac{S_n}{n}\right|<\frac{T_n}{n}\leq \d\right)\,.\]
We denote $\a=-\ln\r(\{0\})/2>0$. The function $I$ is lower semi-continuous, thus there exists a neighbourhood $\Uc$ of $(0,0)$ such that
\[ \forall (x,y) \in \overline{\Uc} \qquad I(x,y)\geq I(0,0)-\frac{\a}{2}=-\left(\ln \r(\{0\})+\frac{\a}{2}\right)\,.\]
We can take $\d$ small enough so that $\{\,(x,y)\in \R^2 : c|x|<y\leq \d\,\} \subset \overline{\Uc}$. We choose $c=\s/\sqrt{\a}$ (which only depends on $\r$). Cram\'er's theorem (see~\cite{DZ}) implies that
\[ \limsupn \frac{1}{n}\ln I_{n,2}\leq \frac{\d}{2c^2}-\inf_{\overline{\Uc}}\,I \leq \frac{\d}{2c^2}+\ln \r(\{0\})+\frac{\a}{2}=\ln \r(\{0\})+\frac{\a}{2}\left(1+\frac{\d}{\s^2}\right)\,.\]
If $\d<\s^2$ then this last expression is smaller than
\[\ln \r(\{0\})+\a=-2\a+\a=-\a\,.\]
Hence, for $n$ large enough,
\[ I_{n,2}\leq \exp\left(-\frac{n\a}{2}\right)\,.\]

\noindent Let us focus now on $I_{n,1}$. We define the random variable $N_n$ by 
\[ N_n=\{\,k\in \{0,\dots,n\}:X_k=0\,\}. \]
We have
\begin{align*}
I_{n,1}=\E\left(e^{S_n^2/(2T_n)}\ind{T_n>0}\ind{T_n/n\leq c|S_n/n|}\right)&=\E\left(e^{S_n^2/(2T_n)}\ind{T_n>0}\ind{T_n\leq c|S_n|}\right)\\
&=\sum_{k=0}^{n-1}\E\left(e^{S_n^2/(2T_n)}\ind{T_n\leq c|S_n|}\ind{N_n=k}\right)
\end{align*}
and, for any $k\in \{0,\dots,n-1\}$, 
\begin{align*}
&\E\left(e^{S_n^2/(2T_n)}\ind{T_n\leq c|S_n|}\ind{N_n=k}\right)\\&=\E\left(\!\!e^{S_n^2/(2T_n)}\ind{T_n\leq c|S_n|}\!\!\!\sum_{1\leq i_1<i_2<\dots<i_k\leq n}\!\!\!\ind{X_{i_1}=0}\,\dots\,\ind{X_{i_k}=0}\,\ind{\forall j\notin \{i_1,\dots,i_k\} \, X_j\neq 0}\!\!\right)\\
&=\sum_{1\leq i_1<i_2<\dots<i_k\leq n}\E\left(e^{S_n^2/(2T_n)}\ind{T_n\leq c|S_n|}\ind{X_{i_1}=0}\,\dots\,\ind{X_{i_k}=0}\,\ind{\forall j\notin \{i_1,\dots,i_k\} \, X_j\neq 0}\right)\!.
\end{align*}
The random variables $X_1,\dots,X_n$ are exchangeable, hence the expectations in the above sum are equal:
\begin{align*}
&\E\left(e^{S_n^2/(2T_n)}\ind{T_n\leq c|S_n|}\ind{N_n=k}\right)\\
&=\!\binom{n}{k}\E\left(e^{S_n^2/(2T_n)}\ind{T_n\leq c|S_n|}\ind{X_{1}\neq 0}\,\dots\,\ind{X_{n-k}\neq 0}\,\ind{X_{n-k+1}= 0}\dots\ind{X_{n}= 0}\!\right)\\
&=\!\binom{n}{k}\E\left(e^{S_{n-k}^2/(2T_{n-k})}\ind{T_{n-k}\leq c|S_{n-k}|}\ind{X_{1}\neq 0}\,\dots\,\ind{X_{n-k}\neq 0}\,\ind{X_{n-k+1}= 0}\dots\ind{X_{n}= 0}\!\right)\!.
\end{align*}
By the independence of $X_1,\dots,X_n$, we have
\begin{align*}
&\E\left(e^{S_n^2/(2T_n)}\ind{T_n\leq c|S_n|}\ind{N_n=k}\right)\\
&=\!\binom{n}{k}\,\prod_{j=n-k+1}^{n}\!\P(X_{j}= 0)\, \E\left(e^{S_{n-k}^2/(2T_{n-k})}\ind{T_{n-k}\leq c|S_{n-k}|}\ind{X_{1}\neq 0}\,\dots\,\ind{X_{n-k}\neq 0}\right)\\
&=\!\binom{n}{k}\r(\{0\})^k (1-\r(\{0\}))^{n-k}\,\E\left(\!e^{S_{n-k}^2/(2T_{n-k})}\ind{T_{n-k}\leq c|S_{n-k}|}\prod_{j=1}^{n-k}\frac{\ind{X_{j}\neq 0}}{\P(X_{j}\neq 0)}\!\right)\!.
\end{align*}
For any $k\in \{1,\dots,n\}$, we set
\[ \overline{u}_k= \E\left(e^{S_{k}^2/(2T_{k})}\ind{T_{k}\leq c|S_{k}|}\prod_{j=1}^{k}\frac{\ind{X_{j}\neq 0}}{\P(X_{j}\neq 0)}\right)\]
so that we have
\[I_{n,1}\!=\!\sum_{k=0}^{n-1}\overline{u}_{n-k} \binom{n}{k}\r(\{0\})^k (1-\r(\{0\}))^{n-k}\!=\!\sum_{k=1}^{n}\overline{u}_{k} \binom{n}{k}\r(\{0\})^{n-k} (1-\r(\{0\}))^{k}.\]
We denote by $\overline{\r}$ the probability measure $\r$ conditioned to $\R\backslash\{0\}$, i.e.,
\[ \overline{\r}=\r(\cdot|\R\backslash\{0\})= \frac{\r(\cdot\cap\R\backslash\{0\})}{1-\r(\{0\})}, \]
so that
\[\forall k\in \{1,\dots,n\} \qquad\overline{u}_k=\int_{\D^{\!*}} e^{kx^2/(2y)}\ind{y\leq c|x|}\,d\widetilde{\nu}_{k,\overline{\r}}(x,y)\,.\]
The measure $\overline{\r}$ is symmetric, $\overline{\r}(\{0\})=0$ and
\[\forall (u,v) \in \R^2 \qquad \overline{\L}(u,v)= \ln \int_{\R}e^{uz+vz^2}\,d\overline{\r}(z)\leq \L(u,v)-\ln(1-\r(\{0\}))\,,\]
thus $(0,0)\in \Dro_{\bar{\L}}$. Moreover the variance of $\overline{\r}$ is $\overline{\s}^2=\s^2(1-\r(\{0\}))^{-1}$ and the closed set $\{\,(x,y)\in \R^2 : y\leq c|x|\,\}$ does not contain $(0,\overline{\s}^2)$. Applying lemma~\ref{TypeVaradhan}, we get
\[ \limsupk \frac{1}{k}\ln \int_{\D^{\!*}} e^{kx^2/(2y)}\ind{y\leq c|x|}\,d\widetilde{\nu}_{k,\overline{\r}}(x,y)<0\,.\]
Thus there exist $\eps_0>0$ and $n_0 \geq 1$ such that
\[ \forall k\geq n_0 \qquad \overline{u}_k \leq \exp(-k\eps_0)\,.\]
For $n>n_0$, we write $I_{n,1}=A_n+B_n$ with
\[A_n=\sum_{k=1}^{n_0}\overline{u}_{k} \binom{n}{k}\r(\{0\})^{n-k} (1-\r(\{0\}))^{k}\]
and
\[B_n=\sum_{k=n_0+1}^{n}\overline{u}_{k} \binom{n}{k}\r(\{0\})^{n-k} (1-\r(\{0\}))^{k}\,.\]
For all $k \geq 1$, we have $\widetilde{\nu}_{k,\overline{\r}}(\D)=1$ thus $\overline{u}_k\leq \exp(k/2)$ and then
\begin{align*}
A_n&\leq \r(\{0\})^n \sum_{k=1}^{n_0}e^{k/2} n^k \left(\r(\{0\})^{-1}-1\right)^k\\
&\leq\r(\{0\})^n n_0 e^{n_0/2} n^{n_0} \max\left(1,\left(\r(\{0\})^{-1}-1\right)^{n_0}\right)\!.
\end{align*}
Moreover
\begin{align*}
B_n&\leq \sum_{k=n_0+1}^{n}e^{-k\eps_0} \binom{n}{k}\r(\{0\})^{n-k} (1-\r(\{0\}))^{k}\\
&\leq \left(\r(\{0\})+e^{-\eps_0}(1-\r(\{0\}))\right)^n.
\end{align*}
Therefore, setting
\[\b=-\ln\left[\r(\{0\})+e^{-\eps_0}(1-\r(\{0\}))\right]>0\,,\]
we have that, for $n$ large enough,
\[ I_{n,1}=A_n+B_n \leq \exp(-n\a)+ \exp(-n\b)\,.\]
We notice that $\eps_0$, $\a$ and $\b$ only depend on $\r$.\medskip

\noindent Finally we set $\g=\min(\a/4,\b/2)$ (which only depends on $\r$). For $\d \in \,]0,\s^2[$ small enough and $n$ large enough, we have
\[ E_{n,\d}\leq I_{n,1}+I_{n,2} \leq \exp(-n\g)\,.\]
This proves the proposition.
\end{proof}
\medskip

\noindent Now we can state the main result of this section, which is the announced refinement of lemma~\ref{TypeVaradhan} and which is essential to the proof of theorem~\ref{MainTheorem}.

\begin{prop} Let $\r$ be a symmetric probability measure on $\R$ with a positive variance $\s^{2}$ and such that $(0,0) \in \Dro_{\L}$. If $A$ is a closed subset of $\R^{2}$ which does not contain $(0,\s^{2})$ then  
\[\limsupn \frac{1}{n}\ln \int_{\D^{\!*}\cap A}\exp\left(\frac{nx^{2}}{2y}\right)\,d\widetilde{\nu}_{n,\r}(x,y)<0\,.\]
\label{TypeVaradhanGen}
\end{prop}

\begin{proof} By proposition~\ref{TypeVaradhan(0,0)}, there exist $\g>0$ and $\d>0$ such that
\[ \limsupn \frac{1}{n}\ln \int_{\D^{\!*}}e^{nx^2/(2y)}\ind{0<y\leq \d}\,d\widetilde{\nu}_{n,\r}(x,y) \leq -\g\,.\]
We set $A_{\d}=\{\,(x,y)\in \D\cap A:y\geq \d\,\}$. We have
\[\D^{\!*}\cap A \subset \{\,(x,y)\in \D^{\!*}:0<y\leq \d\,\} \cup A_{\d}\,.\]
The set $A_{\d}$ is closed, it does not contain $(0,\s^2)$ and $F$ is continuous on it. The usual Varadhan's lemma (see~\cite{DZ}) implies that
\[ \limsupn \frac{1}{n}\ln \int_{A_{\d}}e^{nx^2/(2y)}\,d\widetilde{\nu}_{n,\r}(x,y) <-\inf_{A_{\d}}\, (I-F)\,.\]
As a consequence
\[\limsupn \frac{1}{n}\ln\int_{\D^{\!*}\cap A}\exp\left(\frac{nx^{2}}{2y}\right)\,d\widetilde{\nu}_{n,\r}(x,y)\leq \max\left(-\g\,,\, -\inf_{A_{\d}}\, (I-F)\right)\,.\]
Since $(0,0) \in \Dro_{\L}$, $I$ is a good rate function and $I-F$ attains its minimum on the closed set $A_{\d}$. Since $A_{\d}$ does not contain $(0,\s^2)$, we have
\[\max\left(-\g\,,\, -\inf_{A_{\d}}\, (I-F)\right)<0\]
and the proposition is proved.
\end{proof}\newpage

\section{The Cram\'er condition}
\label{sectionCramer}

\noindent Let $d\geq 1$. For any $z=(a_1+ib_1,\dots,a_d+ib_d) \in \C^{d}$ and $x=(x_1,\dots,x_d) \in \R^d$, we denote
\[\langle z,x\rangle=\sum_{k=1}^d a_k x_k+i\sum_{k=1}^d b_k x_k\,.\]
If $z \in \R^d$ then $\langle z,x\rangle$ is the Euclidean inner product of $z$ and $x$.\medskip

\noindent Let $\nu$ be a non-degenerate probability measure on $\R^{d}$. We denote by $L$ its Log-Laplace and by $J$ its Cram\'er transform. Let $D_L$ and $D_J$ be the domains of $\R^d$ where the functions $L$ and $J$ are respectively finite. We put
\[D_{M}=\{\,z=a+ib \in \C^{d} : a \in D_{L}\,\}\]
and we define the function $M$ by
\[\forall z\in D_{M}\qquad M(z)=\int_{\R^{d}}e^{\langle z,x\rangle}\,d\nu(x)\,.\]
We notice that the function $s \in \R^d \longmapsto \ln M(s)$ is the Log-Laplace $L$ of $\nu$ and that $s\in \R^d \longmapsto M(is)$ is the Fourier transform of $\nu$.\medskip

\noindent One of the key ingredients for proving the main theorem of~\cite{CerfGorny} is the theorem 11 of\cite{CerfGorny} (which is extracted from~\cite{Baldi}). This theorem allows us to express the density of $\nu^{*n}$ as a function of $J$ and, under the condition
\begin{equation}
\forall \a>0 \qquad \sup_{\|s\|\geq \a}\left|M(is)\right|<1,
\label{(C)2} \tag{$C$}
\end{equation}
we can then obtain an asymptotic expansion. The condition $(C)$ is called the Cram\'er condition. In~\cite{MartinLof}, Anders Martin-L\"of uses an approximation of the identity to obtain a similar expression for more general measures on $\R$ satisfying the condition $(C)$, without requiring the existence of a density. \medskip

\noindent In this section we will prove $d$-dimensional analogs of the results of~\cite{MartinLof}.

\subsection{Around the Cram\'er condition}

\noindent We give here a sufficient condition for a measure $\nu$ on $\R^d$ to satisfy the Cram\'er condition $(C)$.

\begin{lem} If there exists $s_0 \neq 0$ such that $|M(is_0)|=1$ then $\nu$ is an arithmetic measure, i.e., there exists $(a,b) \in \R^2$ such that
\[\nu(\{\,x\in \R^d : \langle s_0 , x\rangle \in a+b\Z\,\})=1\,.\]
\end{lem}

\begin{proof} Suppose that $|M(is_0)|=1$ for some $s_0 \neq 0$. Thus
\[1=\left|\int_{\R^{d}}e^{i\langle s_0,x\rangle}\,d\nu(x)\right|\leq \int_{\R^{d}}\,d\nu(x)=1\,.\]
We are in the equality case of this classical inequality, that is, there exists $b_0 \in \R$ such that
\[e^{i\langle s_0,x\rangle}=e^{ib_0}\qquad \nu\,\mbox{a.s.}\,,\]
whence
\[\nu(\{\,x\in \R^d : \langle s_0 , x\rangle \in b_0+2\pi\Z\,\})=1\]
and the lemma is proved.
\end{proof}
\medskip

\noindent Suppose that $\nu$ has a density with respect to the Lebesgue measure. By the Riemann-Lebesgue lemma,
\[|M(is)|=\left|\int_{\R^{d}}e^{i\langle s,x\rangle}\,d\nu(x)\right| \underset{\|s\|\to +\infty}{\longrightarrow}0\,.\]
As a consequence, if $\nu$ does not satisfy $(C)$, then there exists $s_0 \neq 0$, such that $|M(is_0)|=1$. By the previous lemma, $\nu$ is arithmetic. This is absurd. Therefore any probability measure having a density with respect to the Lebesgue measure satisfies $(C)$. Moreover, by the Lebesgue decomposition theorem (see~\cite{Rudin}), a probability measure $\nu$ can be represented as the sum of three components: 
\[\nu=a\,\nu_{ac}+b\,\nu_{d}+c\,\nu_{s}\,,\]
where $\nu_{ac}$ is an absolutely continuous probability measure, $\nu_{d}$ is a discrete probability measure, $\nu_{s}$ is a singular probability measure with no atoms and $a,b,c$ are three non-negative real numbers such that $a+b+c=1$. If $a>0$, we say that $\nu$ has an absolutely continuous component. An absolutely continuous probability measure admits a density, thus we have the following proposition:

\begin{prop} If $\nu$ has an absolutely continuous component then it satisfies the Cram\'er condition $(C)$.
\label{(C)AbsComponent}
\end{prop}

\noindent We end this section by giving the proof of proposition~\ref{a>0Cramer}: we suppose that $\r=a\,\r_{ac}+b\,\r_{d}+c\,\r_{s}$, where $a>0$ and $\r_{ac}$ is a probability measure on $\R$ having a density $f$. We cannot use proposition~\ref{(C)AbsComponent} directly because $\nu_{\r}$ does not have a density. However, we saw in lemma~16\footnote{Actually it is lemma 30 if you refer to the ARXIV version of~\cite{CerfGorny}.} of~\cite{CerfGorny} that, if $\nu_{\r_{ac}}$ denotes the law of $(Z,Z^2)$ where $Z$ is a random variable with distribution $\r_{ac}$, then $\nu_{\r_{ac}}^{*2}$ has the density
\[f_{2}: (x,y)\longmapsto\frac{1}{\sqrt{2y-x^{2}}}\,f\left(\frac{x+\sqrt{2y-x^{2}}}{2}\right)\,f\left(\frac{x-\sqrt{2y-x^{2}}}{2}\right)\,\ind{x^{2}<2y}\,.\]
We can write $\r^{*2}=a^{2}\r_{ac}^{*2}+(1-a^2)\eta$, where $\eta$ is the probability measure on $\R^2$ defined by
\[\eta=\frac{1}{1-a^2}(b^2\r_{d}^{*2}+c^2\r_{s}^{*2}+2ab\,\r_{ac}*\r_{d}+2ac\,\r_{ac}*\r_{s}+2bc\,\r_{d}*\r_{s})\,.\]
We have then
\begin{align*}
&\left|\int_{\R}e^{isz+itz^2}\,d\r(z)\right|^{2}=\left|\int_{\R^{2}}e^{is(x+y)+it(x^2+y^2)}\,d\r(x)\,d\r(y)\right|\\
&\qquad\qquad\leq a^2\left|\int_{\R^{2}}e^{is(x+y)+it(x^2+y^2)}\,d\r^{*2}_{ac}(x,y)\right|+(1-a^2)\left|\int_{\R^{2}}\,d\eta(x,y)\right|\\
&\qquad\qquad\leq a^2\left|\int_{\R^{2}}e^{isu+itv}\,d\nu^{*2}_{\r_{ac}}(u,v)\right|+1-a^2.
\end{align*}
Hence
\[\sup_{\|(s,t)\|\geq \a}\left|\int_{\R}e^{isz+itz^2}\,d\r(z)\right|^{2}\leq a^{2}\sup_{\|(s,t)\|\geq \a}\left|\int_{\R^{2}} e^{isu+itv}f_{2}(u,v)\,du\,dv\right|+1-a^2\,.\]
Proposition~\ref{(C)AbsComponent} implies that the supremum in the right side of the previous inequality is stricly smaller that $1$. This ends the proof of proposition~\ref{a>0Cramer}.

\subsection{An asymptotic relation with the Cram\'er transform}

\noindent We define the function $k$ by
\[\forall x=(x_{1},\dots,x_{d}) \in \R^{d}\qquad k(x)=\prod_{j=1}^{d}\max(1-|x_j|,0)\]
and, for $c>0$, the function $k_c$ by
\[\forall x\in \R^{d}\qquad k_c(c)=\frac{1}{c^{d}}k\left(\frac{x}{c}\right)\,.\]
It is an approximation of the identity on $\R^d$ since the integral of $k$ is equal to~$1$. Finally, for any $n\geq 1$ and $c>0$, we introduce
\[\phi_{n,c}:x\in \R^{d} \longmapsto \int_{\R^{d}}k_{c}(s-nx)\,d\nu^{*n}(s)\,.\]
We notice that $\phi_{n,c}(x)=(k_c* \nu^{*n})(nx)$ for any $x \in \R^{d}$. A standard result on the approximations of the identity says that, if $\nu^{*n}$ has a density $f_{n}$, then
\[\lim_{c \to 0}\int_{\R^{d}}\left|\phi_{n,c}(x)-f_{n}(nx)\right|\,dx = 0\,.\]
This suggests that the asymptotic behaviour of $\phi_{n,c}$ and $\nu^{*n}$ are related, even in the general case when $\nu^{*n}$ does not have a density.

\begin{theo} Let $\nu$ be a non-degenerate probability measure on $\R^d$ such that the interior of $D_{L}$ is not empty. Let $K_{J}$ be a compact subset of $A_{J}$, the admissible domain of $J$. If $\nu$ satisfies the Cram\'er condition
\begin{equation}
\forall \a>0 \qquad \sup_{\|s\|\geq \a}\left|M(is)\right|<1,\tag{$C$}
\end{equation}
then there exists $\g>0$ such that, when $n$ goes to $+\infty$ and $c$ goes to $0$, uniformly over $x \in K_{J}$,
\[\phi_{n,c}(x)=(2\pi n)^{-d/2}\left(\mathrm{det}\, \mathrm{D}_{x}^{2}J\right)^{1/2} e^{-nJ(x)} \left(1+o(1) + O\left(n^{d/2}e^{-\g n}c^{-d}\right)\right)\,.\]
\label{exp(-nI)kc}
\end{theo}
\vspace*{-0.5cm}

\noindent The ideas of the proof of this theorem come from the article~\cite{MartinLof} of Anders Martin-L\"of. It relies also on the following proposition:

\begin{prop} Let $\nu$ be a non-degenerate probability measure on $\R^{d}$ such that the interior of $D_{L}$ is non-empty. Let $A_{J}$ be the admissible domain of $J$.
\smallskip 

\noindent \textbf{(a)} The function $\nabla L$ is a $\Ck{\infty}$-diffeomorphism from $\Dro_{L}$ to $A_{J}$. Moreover
\[A_{J} \subset D_{J}=\{\,x \in \R^{d}:J(x)<+\infty\,\}\,.\]

\noindent \textbf{(b)} Denote by $\l$ the inverse $\Ck{\infty}$-diffeomorphism of $\nabla L$. Then the map $J$ is $\Ck{\infty}$ on $A_{J}$ and for any $x \in A_{J}$,
\[J(x)=\langle x,\l(x)\rangle-L(\l(x))\,,\]
\[\nabla J(x)=(\nabla L)^{-1}(x) = \l(x)\qquad \mbox{and} \qquad\mathrm{D}^{2}_{x}J=\left(\mathrm{D}^{2}_{\l(x)}L\right)^{-1}\,.\]

\noindent \textbf{(c)} If $D_{L}$ is an open subset of $\R^{d}$ then $A_{J}=\Dro_{J}=\Ccro$ where $\Cc$ denotes the convex hull of the support of $\nu$.
\label{Dadmissible}
\end{prop}

\noindent The points $(a)$ and $(b)$ are proved in~\cite{Baldi} and~\cite{Borovkov} and the point $(c)$ in~\cite{CerfGorny}.
\medskip

\noindent We will also need the three following lemmas:

\begin{lem} For any $c>0$ and $z \in \C$,
\[\int_{\R^{d}}e^{\langle x,z\rangle}k_{c}(x)\,dx=\prod_{j=1}^{d}\frac{2(\mathrm{cosh}(cz_j)-1)}{(cz_j)^{2}}\,.\]
Moreover, for any compact $K$ of $\R$, there exists $M>0$ such that
\[\forall s \in \R \qquad \sup_{u \in K}\,\left|\frac{2(\mathrm{cosh}(u+is)-1)}{(u+is)^{2}}\right|\leq \frac{M}{1+s^2}\,.\]
\label{lemk}
\end{lem}

\begin{proof} For any $\z\in \C\backslash\{0\}$,
\begin{align*}
\int_{\R}e^{\z s}\max\left(1-|s|,0\right)\,ds&=\int_{-1}^{1}e^{\z s} (1-|s|)\,ds\\
&=\int_{-1}^{1}e^{\z s}\,ds-2\int_{0}^{1}s\, \mathrm{cosh}(\z s)\,ds\\
&=\frac{2\mathrm{sinh}(\z)}{\z}-2\left(\frac{\mathrm{sinh}(\z)}{\z}-\frac{\mathrm{cosh}(\z)-1}{(\z)^{2}}\right)\\
&=\frac{2(\mathrm{cosh}(\z)-1)}{\z^{2}}.
\end{align*}
and this last function can be extended to a continuous function at $\z=0$.
By Fubini's theorem, we have, for any $c>0$ and $z\in \C^{d}$,
\begin{align*}
\int_{\R^{d}}e^{\langle x,z\rangle}k_{c}(x)\,dx&=\prod_{j=1}^{d}\frac{1}{c}\int_{\R}e^{x_{j}z_{j}}\max\left(1-\left|\frac{x_j}{c}\right|,0\right)\,dx_j\\
&=\prod_{j=1}^{d}\int_{\R}e^{x_{j}cz_{j}}\max\left(1-\left|x_j\right|,0\right)\,dx_j\\
&=\prod_{j=1}^{d}\frac{2(\mathrm{cosh}(cz_j)-1)}{(cz_j)^{2}}.
\end{align*}
Next we define
\[f : (s,u)\in \R\times K \longmapsto \frac{2(1+s^2)(\mathrm{cosh}(u+is)-1)}{(u+is)^{2}}\,.\]
This is a continuous function on $\R\times K$ (at $u=s=0$ it can be extended to a continuous function by setting $f(0,0)=1$). Thus $f$ is bounded over the compact set $[-1,1]\times K$. Moreover, if $|s|>1$ and $u \in K$, we have
\begin{align*}
|f(s,u)|=\frac{2(1+s^2)}{u^2+s^2} \left|\mathrm{cosh}(u+is)-1\right|&\leq 2\left(\frac{1}{s^2}+1\right)\left(\mathrm{cosh}(u)+1\right)\\
&\leq 4\sup_{u \in K}\,(\mathrm{cosh}(u)+1)<+\infty.
\end{align*}
Hence $f$ is bounded over $\R\times K$ by some constant $M>0$. This ends the proof of the lemma.
\end{proof}
\medskip

\begin{lem}[Uniform dominated convergence theorem] Let $\Xc$ be a separable metric space and let $(\O,\Fc,\mu)$ be a measurable space. Let $f$ and $f_{n}$, $n\geq 1$, be real or complex-valued measurable functions defined on $\Xc \times \O$. Suppose that, for any $\o\in \O$, the functions $x \longmapsto f(x,\o)$ and $x \longmapsto f_{n}(x,\o)$, $n \in \N$, are continuous on $\Xc$ and that
\[\sup_{x \in \Xc}|f_{n}(x,\o)-f(x,\o)|\underset{n\to \infty}{\longrightarrow}0\,.\]
Suppose also that there exists a non-negative and integrable function $g$ on $\O$ such that
\[\forall n \in \N\quad \forall x\in \Xc\quad \forall\o\in \O \qquad |f_{n}(x,\o)|\leq g(\o)\,.\]
Then for any $x\in \Xc$, the function $\o \longmapsto f(x,\o)$ is integrable and
\[\sup_{x\in \Xc}\left|\int_{\O}f_{n}(x,\o)\,d\mu(\o)-\int_{\O}f(x,\o)\,d\mu(\o)\right|\underset{n\to \infty}{\longrightarrow}0\,.\]
\label{CVDuniforme}
\end{lem}

\begin{proof} We adapt the proof of the classical dominated convergence theorem in~\cite{Rudin}. Sending $n$ to $+\infty$ in the domination inequality, we get
\[\forall (x,\o) \in \Xc\times \O \qquad |f(x,\o)|\leq g(\o)\,.\] 
This shows that $\o \longmapsto f(x,\o)$ is integrable. For any $n \in \N$, we set
\[h_{n}: \o \longmapsto \sup_{x \in \Xc}|f_{n}(x,\o)-f(x,\o)|\,.\]
For all $n \in \N$ and $\o \in \O$, the function $x\in \Xc \longmapsto |f_{n}(x,\o)-f(x,\o)|$ is continuous and, since $\Xc$ is separable, its supremum is equal to its supremum on a countable dense subset of $\Xc$. Therefore $h_{n}$ is a measurable function. Moreover $(2g-h_{n})_{n\in \N}$ is a sequence of non-negative functions whose limit is the function~$2g$. Fatou's lemma implies that
\begin{align*}
\int_{\O}2g\,d\mu =\int_{\O}\liminfn (2g-h_{n}) \,d\mu&\leq \liminfn \int_{\O} (2g-h_{n}) \,d\mu\\
&=\int_{\O} 2g \,d\mu-\limsupn \int_{\O} h_{n} \,d\mu.
\end{align*}
Since $g$ is integrable, we get that
\[\limsupn \int_{\O} h_{n} \,d\mu \leq 0. \]
Hence $\int_{\O} h_{n} \,d\mu \to 0$ since for any $n \in \N$, $h_{n}$ is a non-negative function. Finally
\begin{align*}
\sup_{x\in \Xc}\left|\int_{\O}f_{n}(x,\o)\,d\mu(\o)-\int_{\O}f(x,\o)\,d\mu(\o)\right|&\leq \sup_{x \in \Xc}\int_{\O}|f_{n}(x,\o)-f(x,\o)|\,d\mu(\o)\\
&\leq\int_{\O} h_{n} \,d\mu \underset{n\to \infty}{\longrightarrow}0.
\end{align*}
and the lemma is proved.
\end{proof}
\medskip 

\begin{lem} If $\nu_2$ is a probability measure on $\R^d$ which satisfies the Cram\'er condition and which is absolutely continuous with respect to a probability measure $\nu_1$ on $\R^d$, then $\nu_1$ satisfies the Cram\'er condition.
\label{CramerAC}
\end{lem}

\noindent We refer to lemma 4 of~\cite{BahadurRao} for a proof.
\medskip

\noindent{\bf Proof of theorem~\ref{exp(-nI)kc}.} Lemma~\ref{lemk} implies that
\[\forall s \in \R^{d} \qquad \widehat{k}_{c}(s)=\prod_{j=1}^{d}\frac{2(1-\mathrm{cos}(cs_j))}{(cs_j)^{2}}\]
and, for any $u\in \R^{d}$, the function $x\longmapsto e^{\langle u,x\rangle}k_{c}(x)$ has the Fourier transform
\[s\in \R^{d}\longmapsto \prod_{j=1}^{d}\frac{2(\mathrm{cosh}(c(u_j+is_j))-1)}{(c(u_j+is_j))^{2}}, \]
which can be rewritten as
\[s\in \R^{d}\longmapsto \prod_{j=1}^{d}\frac{2(1-\mathrm{cos}(c(s_j-iu_j)))}{(c(s_j-iu_j))^{2}}=\widehat{k}_{c}(s-iu)\,.\]
This is an integrable function, thus by the Fourier inversion formula (see~\cite{Rudin}), the Fourier transform of $\smash{s\longmapsto(2\pi)^{-d}\,\widehat{k}_{c}(s-iu)}$ is $y\longmapsto e^{-\langle u,y\rangle}k_{c}(y)$.  
Let $x\in~\!\!\!K_{J}$ and $u\in \R^{d}$. A straightforward computation yields us that the Fourier transform~of
\[s\longmapsto \frac{1}{(2\pi)^{d}}e^{-n\langle x,u+is\rangle}\widehat{k}_{c}(s-iu)\]
is the function $y\longmapsto e^{-\langle u,y\rangle}k_{c}(y-nx)$. We have then
\begin{align*}
\phi_{n,c}(x)&=\int_{\R^d}e^{-\langle u,y\rangle}k_{c}(y-nx)\,e^{\langle u,y\rangle}\,d\nu^{*n}(y)\\
&=\int_{\R^d}\left(\int_{\R^d}e^{i\langle s,y\rangle}\frac{e^{-n\langle x,u+is\rangle}\widehat{k}_{c}(s-iu)}{(2\pi)^{d}}\,ds\right)\,e^{\langle u,y\rangle}\,d\nu^{*n}(y).
\end{align*}
By Fubini's theorem,
\begin{align*}
\phi_{n,c}(x)&=\int_{\R^d}\frac{e^{-n\langle x,u+is\rangle}\widehat{k}_{c}(s-iu)}{(2\pi)^{d}}\,\left(\int_{\R^d}e^{i\langle s,y\rangle}e^{\langle u,y\rangle}\,d\nu^{*n}(y)\right)\,ds\\
&=\int_{\R^{d}}\frac{e^{-n\langle x,u+is\rangle}\widehat{k}_{c}(s-iu)}{(2\pi)^{d}} M(u+is)^{n}\,ds.
\end{align*}
However $x\in A_{J}$ thus, if $\l$ denotes the inverse function of $\nabla L$, then theorem~\ref{Dadmissible} states that
\[J(x)=\langle \l(x),x\rangle-\ln M(\l(x))\,.\]
Replacing $u$ by $\l(x)$ in the previous integral, we get
\[\phi_{n,c}(x)=e^{-nJ(x)} \int_{\R^{d}} e^{-in\langle x,s\rangle}\,\frac{M(\l(x)+is)^{n}}{M(\l(x))^{n}}\,\widehat{k}_{c}(s-i\l(x))\,\frac{ds}{(2\pi)^{d}}\,.\]
We denote by $\mu_{x}$ the measure on $\R^{d}$ such that
\[d\mu_{x}(y)=\frac{e^{\langle x+y,\l(x) \rangle}}{M(\l(x))}\,d\nu(y+x)\,.\]  
Its Fourier transform is the function
\[s \longmapsto e^{-i\langle x,s\rangle}\,\frac{M(\l(x)+is)}{M(\l(x))}\]
so that
\[\phi_{n,c}(x)=e^{-nJ(x)} \int_{\R^{d}} \left(\widehat{\mu}_x(s)\right)^{n}\,\widehat{k}_{c}(s-i\l(x))\,\frac{ds}{(2\pi)^{d}}\,.\]

\noindent For any $x\in K_{J}$, the mean of $\mu_{x}$ is
\[\int_{\R^{d}}y \frac{e^{\langle x+y,\l(x) \rangle}}{\exp M(\l(x))} \,d\nu(y+x)=\int_{\R^{d}}(z-x) \frac{e^{\langle z,\l(x) \rangle}}{M(\l(x))} \,d\nu(z)=\nabla L(\l(x))-x=0\]
and its covariance matrix is $\G_{x}=\mathrm{D}^{2}_{\l(x)}L$ since for $1\leq i,j \leq d$ and $s\in D_{L}$,
\begin{align*}
(\G_{x})_{i,j}&=\frac{\int_{\R^{d}}y_{i}y_{j} e^{\langle \l(x),y+x \rangle}\,d\nu(y+x)}{M(\l(x))}=\frac{\int_{\R^{d}}(z_{i}-x_{i})(z_{j}-x_{j}) e^{\langle \l(x),z \rangle}\,d\nu(z)}{M(\l(x))}\\
&=\frac{\int_{\R^{d}}z_{i}z_{j} e^{\langle \l(x),z \rangle}\,d\nu(z)}{M(\l(x))}-x_{i}x_{j}=\frac{\partial^{2} L}{\partial s_{i}s_{j}}(\l(x)).
\end{align*}
When $t \to 0$, uniformly over $x\in K_{J}$, we have the expansion
\[\widehat{\mu}_x(t)=1-\frac{1}{2}\langle \G_{x}t,t\rangle+o(\|t\|^{2})\,.\]
Indeed the function $(x,t)\longmapsto \widehat{\mu}_x(t)$ is $\Ck{\infty}$ on $A_{J}\times \R^{d}$ (by proposition~\ref{Dadmissible}), thus the Taylor-Lagrange formula guarantees that the remainder term is uniformly controlled over $x\in K_{J}$. Therefore, for any $t \in \R^{d}$, uniformly over $x \in K_{J}$,
\[\widehat{\mu}_x\left(\frac{t}{\sqrt{n}}\right)^{n}\underset{n \to \infty}{\longrightarrow} \exp\left(-\frac{1}{2}\langle \G_{x}t,t\rangle\right)\,.\]
Moreover, for any $c>0$, $n\geq 1$, $t\in \R^{d}$ and $x\in K_{J}$,
\[\widehat{k}_{c}\left(\frac{t}{\sqrt{n}}-i\l(x)\right)=\int_{\R^{d}}f_{c,n}(x,s)\,ds\,,\]
with
\[\forall s \in \R^{d} \qquad f_{c,n}(x,s)=\exp\left(i\frac{c}{\sqrt{n}}\langle s,t\rangle + c\langle s,\l(x) \rangle \right)\,k(s)\,.\]
We have
\[\sup_{x \in K_{J}}\,\left|f_{c,n}(x,s)-k(s)\right|=k(s)\sup_{x \in K_{J}}\,\left|\exp\left(i\frac{c}{\sqrt{n}}\langle s,t\rangle + c\langle s,\l(x) \rangle \right)-1\right|\underset{\gfrac{n\to+\infty}{c\to 0}}{\longrightarrow}0\]
and, for all $s \in \R^{d}$, $x\in K_{I}$, $c\leq 1$ and $n\geq 1$,
\[\left|f_{c,n}(x,s)\right|\leq k(s) \sup_{\gfrac{x \in K_{I}}{t \in [-1,1]^d}}\,\exp\,\langle t,\l(x) \rangle\,.\]
The term on the right defines an integrable function on $\R^{d}$ since $k(s)=0$ for any $s \notin [-1,1]^d$. Thus the uniform dominated convergence theorem (lemma~\ref{CVDuniforme}) states that, for any $t \in \R^{d}$, uniformly over $x \in K_{J}$,
\[\widehat{k}_{c}\left(\frac{t}{\sqrt{n}}-i\l(x)\right) \underset{\gfrac{n\to+\infty}{c\to 0}}{\longrightarrow}1\,.\]

\noindent The functions $x \longmapsto \widehat{\mu}_{x}(t)$ and $x \longmapsto \exp\left(-\langle \G_{x}t,t\rangle/2\right)$, $t\in \R^{d}$, are continuous on $K_{J}$. In order to apply the dominated convergence theorem (the uniform variant), we need to get a uniform domination of the sequence of functions. For~$x \in A_{J}$, $\G_{x}$ is a positive definite symmetric matrix thus $\eps_{x}$, its smallest eigenvalue, is positive. The largest eigenvalue of the inverse of $\G_{x}$ is $\eps_{x}^{-1}$. Therefore, for any $x \in A_{J}$,
\[\eps_{x}=\left(\max\,\{\,\a:\a \mbox{ eigenvalue of }\G^{-1}_{x}\,\}\right)^{-1}=
\left(\sup_{y \neq 0}\frac{\langle \G^{-1}_{x} y, \G^{-1}_{x} y \rangle}{\langle y, y\rangle}\right)^{-1/2}\,.\]
The term on the right is the inverse of the operator norm of the linear application associated to the matrix $\G^{-1}_{x}$. Moreover $x \longmapsto \smash{\G_{x}=\mathrm{D}^{2}_{\l(x)}L}$ is continuous on~$A_{J}$ thus the function $x \longmapsto\eps_{x}$ is continuous. Let us denote by $\eps_{0}$ its minimum on~$K_{J}$. The compactness of $K_{J}$ ensures that $\eps_{0}>0$. The previous expansion implies that there exists $\d>0$ such that
\[\forall (t,x)\in \Bl(0,\d)\times K_{J} \qquad |\widehat{\mu}_x(t)|\leq 1-\frac{1}{2}\left\langle \left(\G_{x}-\frac{\eps_{0}}{2} \mathrm{I}_{d}\right)t ,t \right\rangle\,.\]
The spectral theorem for real symmetric matrices yields that, for any $x\in K_{J}$, the matrix $\G_{x}-\eps_{0}\mathrm{I}_{d}$ is positive symmetric. Thus
\[\forall t \in \R^{d} \qquad \left\langle \left(\G_{x}-\frac{\eps_{0}}{2} \mathrm{I}_{d}\right)t ,t \right\rangle-\frac{\eps_{0}}{2}\|t\|^{2}=\langle (\G_{x}-\eps_{0}\mathrm{I}_{d})t,t \rangle \geq 0\,.\]
It follows that
\[\forall (t,x)\in \Bl(0,\d)\times K_{J} \qquad |\widehat{\mu}_x(t)|\leq 1-\frac{\eps_{0}}{4}\|t\|^{2}\,.\]
Since $1-y\leq e^{-y}$ for all $y\geq 0$, we get
\[\forall n \geq 1 \quad \forall (t,x)\in \Bl(0,\d\sqrt{n})\times K_{J} \qquad \left|\widehat{\mu}_x\left(\frac{t}{\sqrt{n}}\right)\right|^{n}\leq \exp\left(-\frac{\eps_{0}}{4}\|t\|^{2}\right)\,.\]
The right term is integrable and does not depend on $x \in K_{J}$ and $n$.  Moreover $\widehat{k}_c(t)=\widehat{k}(ct)$ for $t\in \R$, and by lemma~\ref{lemk}, the function $\widehat{k}_{c}(\cdot/\sqrt{n}-i\l(x))$ is bounded uniformly over $x\in K_{J}$, $c>0$ and $n\geq 1$. The uniform dominated convergence theorem (lemma~\ref{CVDuniforme}) implies that, uniformly over $x\in K_{J}$,
\[\int_{\|t\|<\d \sqrt{n}}\!\!\widehat{\mu}_{x}\left(\,\frac{t}{\sqrt{n}}\!\right)^{n}\,\widehat{k}_{c}\!\left(\!\frac{t}{\sqrt{n}}-i\l(x)\!\right)\,dt\underset{\gfrac{n\to+\infty}{c\to 0}}{\longrightarrow}\int_{\R^{d}}\!\!\exp\!\left(\!-\frac{1}{2}\left\langle \left(\mathrm{D}^{2}_{\l(x)}L\!\right) t,t\right\rangle\right)\,dt\,.\]
Moreover this second integral is equal to $(2\pi)^{d/2} \left(\mathrm{det}\,\G_{x}\right)^{-1/2}$ and proposition~\ref{Dadmissible} guarantees that, for $x \in A_{J}$, $\mathrm{D}_{\l(x)}^{2}L$ is the inverse matrix of $\mathrm{D}_{x}^{2}J$. Therefore, when $n\to \infty$ and $c\to 0$, uniformly over $x \in K_{J}$,
\begin{align*}
\int_{\|t\|<\d}\widehat{\mu}_{x}(t)^{n}\,\widehat{k}_{c}(s-i\l(x))\,ds&=n^{-d/2}\int_{\|t\|<\d \sqrt{n}}\widehat{\mu}_{x}\left(\frac{t}{\sqrt{n}}\right)^{n}\,\widehat{k}_{c}\left(\!\frac{t}{\sqrt{n}}-i\l(x)\!\right)\,dt\\
&\sim \left(\frac{2\pi}{n}\right)^{d/2} \left(\mathrm{det}\, \mathrm{D}_{x}^{2}J\right)^{1/2}.
\end{align*}

\noindent Let us consider now the remaining integral
\[\int_{\|t\|\geq\d}\widehat{\mu}_{x}(t)^{n}\,\widehat{k}_{c}(s-i\l(x))\,ds\,,\]
the rest of the integral. The measure $\nu$ satisfies the Cram\'er condition and $\nu$ is absolutely continuous with respect to $\mu_x$. By lemma~\ref{CramerAC}, we get that $\mu_x$ also satisfies the Cram\'er condition: 
\[\sup_{\|s\|\geq \d}\left|\widehat{\mu}_x(s)\right|<1. \]
Therefore, by the compactness of $K_{J}$,
\[\sup_{x \in K_{J}}\sup_{\|s\|\geq \d}|\widehat{\mu}_{x}(s)|=e^{-\g}<1\,,\]
for some $\g>0$. As a consequence
\[\sup_{x\in K_{J}}\left|\int_{\|s\|\geq\d}\widehat{\mu}_{x}(s)^{n}\,\widehat{k}_{c}(s-i\l(x))\,ds\right|\leq e^{-n\g} \int_{\R^{d}} \sup_{x\in K_{J}}\,\widehat{k}_{c}(s-i\l(x))\,ds\,.\]
By lemma~\ref{lemk}, we have
\[\int_{\R^{d}} \sup_{x\in K_{J}}\,\widehat{k}_{c}(s-i\l(x))\,ds=O\left(\prod_{j=1}^{d}\int_{\R^{d}}\frac{1}{1+(cs_j)^2}\,ds_j\right)=O\left(\frac{1}{c^d}\right)\,.\]
Finally, when $n\to +\infty$ and $c\to 0$,
\begin{align*}
\phi_{n,c}(x)&=\frac{e^{-nJ(x)}}{(2\pi)^{d}} \left(\left(\frac{2\pi}{n}\right)^{d/2} \left(\mathrm{det}\, \mathrm{D}_{x}^{2}J\right)^{1/2}(1+o(1))+O\left(e^{-n\g}c^{-d}\right)   \right)\\
&=(2\pi n)^{-d/2}\left(\mathrm{det}\, \mathrm{D}_{x}^{2}J\right)^{1/2} e^{-nJ(x)} \left(1+o(1) + O\left(n^{d/2}e^{-\g n}c^{-d}\right)\right)\!.
\end{align*}
The boundedness of the function $x\longmapsto \left(\mathrm{det}\, \mathrm{D}_{x}^{2}J\right)^{1/2}$ on $K_{J}$ and the previous study show us that this expansion is uniform over $x \in K_{J}$. This ends the proof of theorem~\ref{exp(-nI)kc}.\qed

\section{Proof of theorem~\ref{MainTheorem}}
\label{ProofMainTheo}

\noindent In this section we use first proposition~\ref{TypeVaradhanGen} to prove the law of large numbers under $\widetilde{\mu}_{n,\r,g}$. Next, in order to prove the fluctuations theorem, we use Laplace's method: to this end, we introduce an integral with the approximation of the identity of section~\ref{sectionCramer}. Then proposition~\ref{exp(-nI)kc} gives the expansion of this integral. The technical part of the proof is to show that the remaining terms are negligible.\medskip

\noindent Suppose that $\r$ is a symmetric probability measure on $\R$ with positive variance~$\s^2$ and such that
\[\exists v_{0}>0 \qquad \int_{\R}e^{v_{0}z^{2}}\,d\r(z)<+\infty\,.\]
The fact that $g(u)\sim u^2/2$ in the neighbourhood of $0$ implies that $F_g$ is positive on some open neighbourhood~$\Vc$ of $(0,\s^2)$, which is included in $\D^{\!*}$. We have~then
\[Z_{n,g}=\int_{\D^{\!*}}\exp\left(nF_{\!g}(x,y)\right)\,d\widetilde{\nu}_{n,\r}(x,y)\geq \widetilde{\nu}_{n,\r}(\Vc)\,.\]
The large deviations principle satisfied by $(\widetilde{\nu}_{n,\r})_{n\geq 1}$ implies that
\[\liminfn \frac{1}{n} \ln Z_{n,g}\geq \liminfn \frac{1}{n} \ln \widetilde{\nu}_{n,\r}(\Vc)\geq -\inf_{(x,y)\in \Vc}\,I(x,y)=0\,.\]
We denote by $\theta_{n,\r,g}$ the distribution of $(S_{n}/n,T_{n}/n)$ under $\widetilde{\mu}_{n,\r,g}$. Let $\Uc$ be an open neighbourhood of $(0,\s^{2})$ in $\R^{2}$. Since $F_{\!g}\leq F$, the results of section~\ref{Preliminary} and proposition~\ref{TypeVaradhanGen} imply that
\begin{multline*}
\limsupn \frac{1}{n}\ln \theta_{n,\r,g}(\Uc^{c})\leq\limsupn \frac{1}{n}\ln \int_{\D^{\!*}\cap \,\Uc^{c}}\exp\left(nF_{\!g}(x,y)\right)\,d\widetilde{\nu}_{n,\r}(x,y)\\
-\liminfn \frac{1}{n}\ln Z_{n,g}<0.
\end{multline*}
Hence there exist $\eps >0$ and $n_{0}\in \N$ such that 
\[ \forall n>n_0 \qquad\theta_{n,\r}(\Uc^{c}) \leq \exp(-n\eps)\,.\]
Thus, for each open neighbourhood $\Uc$ of $(0,\s^{2})$,
\[\lim_{n \to +\infty}\widetilde{\mu}_{n,\r,g}\left(\left(\frac{S_{n}}{n},\frac{T_{n}}{n}\right)\in \,\Uc^{c}\right)=0\,.\]
This means that, under $\widetilde{\mu}_{n,\r,g}$, $(S_{n}/n,T_{n}/n)$ converges in probability to $(0,\s^{2})$.\medskip

\noindent We suppose in addition that $g$ has a fourth derivative at~$0$ and that $\r$ satisfies 
\begin{equation}
\forall \a>0 \qquad \sup_{\|(s,t)\|\geq \a}\left|\int_{\R}e^{isz+itz^2}\,d\r(z)\right|<1.\tag{$C$}
\end{equation}
This is the Cram\'er condition for $\nu_{\r}$. Let us prove that, under $\widetilde{\mu}_{n,\r,g}$,
\[\frac{S_{n}}{n^{3/4}} \overset{\Lc}{\underset{n \to \infty}{\longrightarrow}} \left(\frac{4(\mu_{4}+m_4\s^4)}{3\s^{4}}\right)^{1/4}\G\left(\frac{1}{4}\right)^{-1} \exp\left(-\frac{\mu_{4}+m_4\s^4}{12 \s^{8}}s^{4}\right)\,ds\,.\]
This is equivalent to the convergence announced in theorem~\ref{MainTheorem}. For $u \in \R$, we define
\begin{multline*}
E_{n}(u)=\int_{\R^{n}}\exp\left(iu\,\frac{x_{1}+\dots+x_{n}}{n^{3/4}}
+ng\left(\frac{x_1+\dots+x_n}{\sqrt{n(x_1^2+\dots+x_n^2)}}\right)\right)\\
\times\ind{\{x_{1}^{2}+\dots+x_{n}^{2}>0\}}\,\prod_{j=1}^{n}d\r(x_{j}).
\end{multline*}
Let us notice that $Z_{n,g}=E_{n}(0)$ and that
\[\E_{\tilde{\mu}_{n,\r}}\left[\exp\left(iu\,\frac{S_n}{n^{3/4}}\right)\right]=\frac{E_{n}(u)}{E_{n}(0)}\,.\]
By Paul Levy's theorem, in order to obtain the convergence in law stated in theorem~\ref{MainTheorem}, it is necessary and sufficient to prove that, for any $u \in \R$, the sequence $(E_{n}(u)/E_{n}(0))_{n\geq 1}$ converges towards
\[\frac{\displaystyle{\int_{\R}\exp\left(iux-\frac{(\mu_{4}+m_4\s^4)x^{4}}{12\s^{8}}\right)\,dx}}{\displaystyle{\int_{\R}\exp\left(-\frac{(\mu_{4}+m_4\s^4)x^{4}}{12\s^{8}}\right)\,dx}}\,.\]
To this end, we will compute the expansion of $E_n(u)$, $n \geq 1$, $u\in \R$. We denote by $\widetilde{\nu}_{n,\r}$ the law of $(S_{n}/n,T_{n}/n)$ under $\r^{\otimes n}$. We have
\[\forall u\in \R \qquad E_{n}(u)=\int_{\D^{\!*}}\exp\left(iuxn^{1/4}+nF_{\!g}(x,y)\right)\,d\widetilde{\nu}_{n,\r}(x,y)\,.\]

\noindent Let $u\in \R$ and $\d>0$. We denote by $\Bl_{\d}$ the open ball in $\R^2$ of radius $\d$ centered at $(0,\s^{2})$. We choose $\d$ small enough so that $\Bl_{\d}$ is included in $K_I$, a compact subset of $A_I\subset \D^{\!*}$.
We define
\[f_n : (x,y)\in \R^{2} \longmapsto \exp(iuxn^{1/4})\,.\]
For all $n\geq 1$, we write $E_{n}(u)=A_{n}+B_{n}$ with
\[A_{n}=\int_{\Bl_{\d}}f_{n}\,e^{nF_{\!g}}\,d\widetilde{\nu}_{n,\r} \qquad \mbox{and} \qquad B_{n}=\int_{\left(\Bl_{\d}\right)^{c}\cap \D^{\!*}}f_{n}\,e^{nF_{\!g}}\,d\widetilde{\nu}_{n,\r}. \]
First, since $F_{\!g}\leq F$, proposition~\ref{TypeVaradhanGen} implies that there exists $\eps_{0}>0$ such that, for $n$ large enough,
\[|B_{n}|\leq \exp(-n\eps_{0})\,.\]
We next compute the expansion of $A_n$, using the results of the last section. We define the function $k$ by
\[\forall (x,y) \in \R^{2}\qquad k(x,y)=\max(1-|x|,0) \,\times\,\max(1-|y|,0)\]
and, for $c>0$, we define $k_c$ by
\[\forall (x,y) \in \R^{2}\qquad k_c(x,y)=\frac{1}{c^{2}}k\left(\frac{x}{c},\frac{y}{c}\right)\,.\]
We put
\[A_{n,c,1}=\int_{\R^{2}} k_{c/n}*\left(f_{n}e^{nF_{\!g}}\ind{\Bl_{\d}}\right)(s,t)\,d\widetilde{\nu}_{n,\r}(s,t)\]
and $A_{n,c,2}=A_{n}-A_{n,c,1}$. Fubini's theorem implies that
\begin{align*}
&A_{n,c,1}=\int_{\R^{2}} k_{c/n}*\left(f_{n}\,e^{nF_{\!g}}\ind{\Bl_{\d}}\right)\left(\frac{s}{n},\frac{t}{n}\right)\,d\nu_{\r}^{*n}(s,t)\\
&\quad=\int_{\R^{2}} \left(\int_{\R^{2}}k_{c/n}\left(\frac{s}{n}-x,\frac{t}{n}-y\right)f_{n}(x,y)\,e^{nF_{\!g}(x,y)}\ind{\Bl_{\d}}(x,y)\,dx\,dy\right)\,d\nu_{\r}^{*n}(s,t)\\
&\quad=\int_{\R^{2}} f_n(x,y)\,e^{nF_{\!g}(x,y)}\ind{\Bl_{\d}}(x,y)\left(\int_{\R^{2}}n^{2}k_{c}\left(s-nx,t-ny\right)\,d\nu_{\r}^{*n}(s,t)\right)\,dx\,dy\\
&\quad=n^{2}\int_{\Bl_{\d}} f_n(x,y)\,e^{nF_{\!g}(x,y)}\phi_{n,c}(x,y)\,dx\,dy,
\end{align*}
where
\[\forall (x,y) \in \R^{2} \qquad \phi_{n,c}(x,y)=\int_{\R^{2}}k_{c}\left(s-nx,t-ny\right)\,d\nu_{\r}^{*n}(s,t)\,.\]
We denote
\[H_{n,c} : (x,y) \in \R^{2} \longmapsto n e^{nI(x,y)}\phi_{n,c}(x,y)\,.\]
Hence
\[A_{n,c,1}=n \int_{\Bl_{\d}} f_n(x,y)\,e^{-n(I-F_{\!g})(x,y)}\,H_{n,c}(x,y)\,dx\,dy. \]
The measure $\nu_{\r}$ satisfies the Cram\'er condition, thus, by theorem~\ref{exp(-nI)kc}, there exists $\g>0$ such that, when $n$ goes to $+\infty$ and $c$ goes to $0$, uniformly over $(x,y) \in K_I$,
\[H_{n,c}(x,y)=\frac{1}{2\pi} \left(\mathrm{det}\, \mathrm{D}_{(x,y)}^{2}I\right)^{1/2}\left(1+o(1) + O\left(n e^{-\g n}c^{-2}\right)\right)\,.\]
We suppose that
\[\eps_{n,c}=n e^{-\g n}c^{-2}\underset{\gfrac{n\to \infty}{c\to 0}}{\longrightarrow}0\,.\] 
Then, uniformly over $(x,y) \in K_I$,
\[H_{n,c}(x,y)\underset{\gfrac{n\to \infty}{c\to 0}}{\longrightarrow}\frac{1}{2\pi} \left(\mathrm{det}\, \mathrm{D}_{(x,y)}^{2}I\right)^{1/2}\,.\]
We denote
\[\Bl_{\d,n}=\{\,(x,y) \in \R^{2} : \|(x n^{-1/4},yn^{-1/2})\|\leq \d\,\}\,,\]
where $\|\cdot\|$ is the euclidean norm on $\R^2$. Let us make the change of variable given by $(x,y) \longmapsto \left(xn^{-1/4},yn^{-1/2}+\s^2\right)$ with Jacobian $n^{-3/4}$:
\begin{multline*}
A_{n,c,1}=n^{1/4} \int_{\Bl_{\d,n}}\exp\left(iux-n(I-F_{\!g})\left(xn^{-1/4},yn^{-1/2}+\s^2\right)\right)\\ \times
H_{n,c}\left(xn^{-1/4},yn^{-1/2}+\s^2\right)\,dx\,dy.
\end{multline*}
We check now that we can apply the dominated convergence theorem to this integral. The uniform expansion of $H_{n,c}$ means that for any $\a>0$, there exist $n_{0} \in \N$ and $c_0>0$ such that
\[(x,y) \in K_{I} \quad n\geq n_{0} \quad c\leq c_{0} \quad\!\Longrightarrow\! \quad\left|H_{n,c}(x,y)\,2\pi \left(\mathrm{det}\, \mathrm{D}_{(x,y)}^{2}I\right)^{-1/2} -1\right|\leq \a\,.\]
If $(x,y) \in \Bl_{\d,n}$, then $(x_{n},y_{n})=(xn^{-1/4},yn^{-1/2}+\s^{2}) \in \Bl_{\d} \subset K_{I}$, thus for all $n\geq n_{0}$, $c \leq c_0$ and $(x,y) \in \Bl_{\d,n}$,
\[\left|H_{n,c}\left(\frac{x}{n^{1/4}},\frac{y}{\sqrt{n}}+\s^{2}\right)2\pi \left(\mathrm{det}\, \mathrm{D}_{(x_{n},y_{n})}^{2}I\right)^{-1/2} -1\right|\leq \a\,.\]
Moreover $(x_{n},y_{n}) \to (0,\s^{2})$ thus, by continuity,
\[\left(\mathrm{D}_{(x_{n},y_{n})}^{2}I\right)^{-1/2} \underset{n\to +\infty}{\longrightarrow} \left(\mathrm{D}_{(0,\s^{2})}^{2}I \right)^{-1/2} = \left(\mathrm{D}_{(0,0)}^{2}\L\right)^{1/2}, \]
whose determinant is equal to $\sqrt{\s^{2}(\mu_{4}-\s^{4})}$.
Therefore
\[\ind{\Bl_{\d,n}}(x,y)\,H_{n,c}\left(\frac{x}{n^{1/4}},\frac{y}{\sqrt{n}}+\s^{2}\right)\underset{\gfrac{n\to \infty}{c\to 0}}{\longrightarrow} \left(4 \pi^{2} \s^{2}(\mu_{4}-\s^{4})\right)^{-1/2}\,.\]
We proved in section~\ref{Preliminary} that, when $(x,y)$ goes to $(0,\s^{2})$,
\[I(x,y)-F_{\!g}(x,y)\sim\frac{(\mu_{4}+m_4\s^4)x^{4}}{12\s^{8}}+\frac{(y-\s^{2})^{2}}{2(\mu_{4}-\s^{4})}\,.\]
It follows that
\[n(I-F_{\!g})\left(\frac{x}{n^{1/4}},\frac{y}{\sqrt{n}}+\s^{2}\right) \underset{n\to +\infty}{\longrightarrow} \frac{(\mu_{4}+m_4\s^4)x^{4}}{12\s^{8}}+\frac{y^{2}}{2(\mu_{4}-\s^{4})}\,.\]
Let us check that the integrand is dominated by an integrable function, which is independent of $n$. The function
\[(x,y) \longmapsto \left(\mathrm{D}_{(x,y)}^{2}I\right)^{-1/2}\]
is bounded on $\Bl_{\d}$ by some $M_{\d}>0$. The uniform expansion of $H_{n,c}$ implies that for all $(x,y) \in \Bl_{\d}$, $H_{n,c}(x,y)\leq C_{\d}$ for some constant $C_{\d}>0$. Finally, it follows from the above expansion of the proposition that, for $\d>0$ small enough,
\[\forall (x,y) \in \Bl_{\d} \qquad G(x,y)=I(x,y)-F_{\!g}(x,y) \geq \frac{(\mu_{4}+m_4\s^4)x^{4}}{24\s^{8}}+\frac{(y-\s^{2})^{2}}{4(\mu_{4}-\s^{4})}\]
and thus, for $\d$ small enough, for any $(x,y) \in  \R^{2}$, $n\geq n_{0}$ and $c\leq c_0$,
\begin{multline*}
\ind{\Bl_{\d,n}}(x,y)\exp\left(-n(I-F_{\!g})\left(\frac{x}{n^{1/4}},\frac{y}{\sqrt{n}}+\s^{2}\right)\right)H_{n,c}\left(\frac{x}{n^{1/4}},\frac{y}{\sqrt{n}}+\s^{2}\right) \\
\leq C_{\d}\exp\left(-\frac{(\mu_{4}+m_4\s^4)x^{4}}{24\s^{8}}-\frac{y^{2}}{4(\mu_{4}-\s^{4})}\right)\!.
\end{multline*}
and the right term is an integrable function on $\R^{2}$. It follows from the dominated convergence theorem that, when $n$ goes to $+\infty$ and $c$ goes to $0$, then $n^{-1/4}A_{n,c,1}$ converges to
\[\int_{\R^{2}}\frac{\exp(iux)}{\sqrt{2 \pi \s^{2}}\sqrt{2 \pi (\mu_{4}-\s^{4})}}\exp\left(-\frac{(\mu_{4}+m_4\s^4)x^{4}}{12\s^{8}}-\frac{y^{2}}{2(\mu_{4}-\s^{4})}\right)dx\,dy\,.\]
By Fubini's theorem, we get
\[A_{n,c,1}\underset{\gfrac{n\to \infty}{c\to 0}}{\sim}\frac{n^{1/4}}{\sqrt{2\pi \s^{2}}}\int_{\R}\exp\left(iux-\frac{(\mu_{4}+m_4\s^4)x^{4}}{12\s^{8}}\right)\,dx\,.\]
Now we deal with $A_{n,c,2}$. We will introduce an indicator function in order to simplify the expression of $A_{n,c,2}$. We put $\a=\d/(2\sqrt{2})$ and 
\[A_{n,c,3}=\int_{\Bl_{\a}}\left[f_{n}(s,t)\,e^{nF_{\!g}(s,t)}\ind{\Bl_{\d}}(s,t)- k_{c/n}*\left(f_{n}e^{nF_{\!g}}\ind{\Bl_{\d}}\right)(s,t)\right]\,d\widetilde{\nu}_{n,\r}(s,t)\,,\]
\[A_{n,c,4}=\int_{\left(\Bl_{\a}\right)^{c}}f_{n}(s,t)\,e^{nF_{\!g}(s,t)}\ind{\Bl_{\d}}(s,t)\,d\widetilde{\nu}_{n,\r}(s,t)\,,\]
\[A_{n,c,5}=\int_{\left(\Bl_{\a}\right)^{c}}k_{c/n}*\left(f_{n}e^{nF_{\!g}}\ind{\Bl_{\d}}\right)(s,t)\,d\widetilde{\nu}_{n,\r}(s,t)\,,\]
so that $A_{n,c,2}=A_{n,c,3}+A_{n,c,4}-A_{n,c,5}$. Since $\Bl_{\d}\subset \D^{\!*}$ and $F_{\!g}\leq F$, we have
\[\left|A_{n,c,4}\right|\leq \int_{\left(\Bl_{\a}\right)^{c}\cap \D^{\!*}}e^{nF}\,d\widetilde{\nu}_{n,\r}\]
and proposition~\ref{TypeVaradhanGen} ensures that there exists $\eps_{1}>0$ such that, for $n$ large enough, 
\[A_{n,c,4}\underset{\gfrac{n\to \infty}{c\to 0}}{=}O\left(\exp(-n\eps_{1})\right)\,.\]

\noindent Until now we used the standard techniques of Laplace's method (cf. the proof of the main result of~\cite{CerfGorny}) together with an approximation of the identity. The computation of the expansion of $A_{n,c,3}$ and $A_{n,c,5}$ is the technical part of this proof. 

\begin{lem} If $\d$, $c/n$ and $cn^{1/4}$ are small enough, then
\[A_{n,c,3}\underset{\gfrac{n\to \infty}{c\to 0}}{=}o\left(E_{n}(0)\right)\,,\]
\[A_{n,c,5}\underset{\gfrac{n\to \infty}{c\to 0}}{=}O\left(\int_{\left(\Bl_{\a}\right)^{c}}e^{nF(s,t)}\,d\widetilde{\nu}_{n,\r}(s,t)\right)\,.\]
\label{lemA_nc34}
\end{lem}

\noindent Suppose that lemma~\ref{lemA_nc34} has been proved. Then proposition~\ref{TypeVaradhanGen} ensures that there exists $\eps_{2}>0$ such that, for $n$ large enough, 
\[A_{n,c,5}\underset{\gfrac{n\to \infty}{c\to 0}}{=}O\left(\exp(-n\eps_{2})\right)\,.\]
We put now together the previous estimates in order to conclude. We take $c=1/n$ so that $c$, $n e^{-\g n}c^{-2}$ and $cn^{1/4}$ go to $0$ when $n \to +\infty$. For $\d$ small enough, when $n$ goes to $+\infty$, we have
\begin{multline*}
A_{n}=\frac{n^{1/4}}{\sqrt{2\pi \s^{2}}}\int_{\R}\exp\left(iux-\frac{(\mu_{4}+m_4\s^4)x^{4}}{12\s^{8}}\right)\,dx\,(1+o(1))\\
+o\left(E_{n}(0)\right)+O\left(e^{-n\eps_{1}}+e^{-n\eps_{2}}\right).
\end{multline*}
Finally
\[e^{-n\eps_{0}}+e^{-n\eps_{1}}+e^{-n\eps_{2}}\underset{n\to \infty}{=}o\left(\frac{n^{1/4}}{\sqrt{2\pi \s^{2}}}\int_{\R}\exp\left(iux-\frac{(\mu_{4}+m_4\s^4)x^{4}}{12\s^{8}}\right)\,dx\right)\]
thus $E_{n}(u)=A_{n}+B_{n}$ is equal to
\[\frac{n^{1/4}}{\sqrt{2\pi \s^{2}}}\int_{\R}\exp\left(iux-\frac{(\mu_{4}+m_4\s^4)x^{4}}{12\s^{8}}\right)\,dx\,(1+o(1))+o\left(E_{n}(0)\right)\,.\]
Hence
\[E_{n}(0)\sim \frac{n^{1/4}}{\sqrt{2\pi \s^{2}}}\int_{\R}\exp\left(iux-\frac{(\mu_{4}+m_4\s^4)x^{4}}{12\s^{8}}\right)\,dx\,.\]
Therefore
\[\frac{E_{n}(0)}{E_{n}(0)}\underset{n\to +\infty}{\longrightarrow} \frac{\displaystyle{\int_{\R}\exp\left(iux-\frac{(\mu_{4}+m_4\s^4)x^{4}}{12\s^{8}}\right)\,dx}}{\displaystyle{\int_{\R}\exp\left(-\frac{(\mu_{4}+m_4\s^4)x^{4}}{12\s^{8}}\right)\,dx}}\,.\]
This ends the proof of theorem~\ref{MainTheorem}.\medskip

\noindent We still have to prove the expansions of $A_{n,c,3}$ and $A_{n,c,5}$ stated in lemma~\ref{lemA_nc34}.\medskip

\noindent{\bf Proof of Lemma~\ref{lemA_nc34}.} For $(s,t)\in \Bl_{\a}$, if we have $k_{c/n}(x-s,y-t)\neq 0$, then
\[1-|n(x-s)/c|>0 \qquad \mbox{and} \qquad 1-|n(y-t)/c|>0\]
and thus, for $c/n<\a$,
\[|x|\leq |x-s|+|s|< \frac{c}{n}+\frac{\d}{2\sqrt{2}}< \frac{\d}{\sqrt{2}}\,,\]
\[|y-\s^2|\leq |y-t|+|t-\s^2|< \frac{c}{n}+\frac{\d}{2\sqrt{2}}< \frac{\d}{\sqrt{2}}\,.\]
Hence $(x,y)\in \Bl_{\d}$ and
\[\forall (s,t) \in\Bl_{\a} \qquad  k_{c/n}(x-s,y-t)=k_{c/n}(x-s,y-t)\ind{\Bl_{\d}}(x,y)\,.\]
This implies that
\[\ind{\Bl_{\a}}\times \left(k_{c/n}*\left(f_{n}e^{nF_{\!g}}\ind{\Bl_{\d}}\right)\right)=\ind{\Bl_{\a}}\times\left( k_{c/n}*\left(f_{n}e^{nF_{\!g}}\right)\right)\,.\]
We have shown that, for $c/n<\a$,
\[A_{n,c,3}=\int_{\R^{2}}\ind{\Bl_{\a}}(s,t)\left[f_{n}(s,t)\,e^{nF_{\!g}(s,t)}- k_{c/n}*\left(f_{n}e^{nF_{\!g}}\right)(s,t)\right]\,d\widetilde{\nu}_{n,\r}(s,t)\,.\]
Let $(s,t)\in \Bl_{\a}$. We have
\begin{align*}
&\left[f_{n}\,e^{nF_{\!g}}- k_{c/n}*\left(f_{n}e^{nF_{\!g}}\right)\right](s,t)\\
&\qquad=\int_{\R^{2}}\left(f_{n}(s,t)e^{nF_{\!g}(s,t)}-f_{n}(s-x,t-y)e^{nF_{\!g}(s-x,t-y)}\right)\,k_{c/n}(x,y)\,dx\,dy\\
&\qquad=e^{nF_{\!g}(s,t)}f_{n}(s,t)\int_{\R^{2}}\left(1-e^{n\Psi_{s,t,n}(cx/n,cy/n)}\right)\,k(x,y)\,dx\,dy\\
&\qquad=e^{nF_{\!g}(s,t)}f_{n}(s,t)\int_{[-1,1]^2}\left(1-e^{n\Psi_{s,t,n}(cx/n,cy/n)}\right)\,k(x,y)\,dx\,dy,
\end{align*}
with, for each $(x,y)\in \R^{2}$,
\[\Psi_{s,t,n}(x,y)=F_{\!g}(s-x,t-y)-F_{\!g}(s,t)-iuxn^{1/4}\,.\]
By hypothesis, the function $g$ has a fourth derivative at~$0$ thus $g$ is $\Ck{1}$ in a neighbourhood of~$0$. As a consequence $F_{\!g}$ is $\Ck{1}$ in a neighbourhood of $(0,\s^2)$. Hence the mean value inequality implies that there exist $r>0$ and $M>0$ such that, for any $(s,t) \in \Bl_{r}$ and $(x,y)\in [-1,1]^2$,
\[|x|<r \quad \mbox{and} \quad |y|<r \quad\Longrightarrow \quad |F_{\!g}(s-x,t-y)-F_{\!g}(s,t)|\leq M \|(x,y)\|\,.\]
If $\d$ is small enough (so that $\a\leq r$) and $c\leq rn$ then, for any $(s,t)\in \Bl_{\a}$ and $(x,y)\in [-1,1]^2$,
\begin{align*}
\left|n\Psi_{s,t,n}\left(\frac{cx}{n},\frac{cy}{n}\right)\right|&\leq Mn\left\|\left(\frac{cx}{n},\frac{cy}{n}\right)\right\|+n\left|u\,\frac{cx}{n}\right|n^{1/4}\\
&\leq M \sqrt{2}\,c  +   |u|\,c\, n^{1/4}.
\end{align*}
By applying the mean value inequality to the function $(x,y) \in \R^2 \longmapsto e^{x+iy}$, we prove that, if $z\in \C$ has a small enough real part, then $\left|1-e^{z}\right|\leq 2 |z|$. Therefore, if $cn^{1/4}$ goes to $0$, then, for any $(s,t)\in \Bl_{\a}$, uniformly over $(x,y)\in [-1,1]^2$,
\[\left|1-e^{n\Psi_{s,t,n}(cx/n,cy/n)}\right|\leq 2M \sqrt{2}\,c  +   2|u|\,c\, n^{1/4}=o(1)\,.\]
Hence, if $\d$, $c/n$ and $cn^{1/4}$ are small enough, then
$A_{n,c,3}=o(E_n(0))$ when $n\to \infty$ and $c\to 0$. Next, for $(s,t) \in \R^2$, we have
\[k_{c/n}*\left(f_{n}e^{nF_{\!g}}\ind{\Bl_{\d}}\right)(s,t)=\int_{[-c/n,c/n]^2}k_{c/n}(x,y)\left(f_{n}e^{nF_{\!g}}\ind{\Bl_{\d}}\right)(s-x,t-y)\,dx\,dy\,.\]
We suppose that $\|s,t-\s^2\|>\d+\sqrt{2}c/n$. For $|x|\leq c/n$ and $|y|\leq c/n$, we have then
\[\|(s-x,t-y)-(0,\s^{2})\|\geq \|s,t-\s^2\|-\|x,y\|>\d+\sqrt{2}c/n-\sqrt{(c/n)^2+(c/n)^2}>\d\]
so that $\ind{\Bl_{\d}}(s-x,t-y)=0$ and then
\[k_{c/n}*\left(f_{n}e^{nF_{\!g}}\ind{\Bl_{\d}}\right)(s,t)=0\,.\]
If $c/n$ is small enough so that $\d+\sqrt{2}c/n\leq 2\d$ then
\[k_{c/n}*\left(f_{n}e^{nF_{\!g}}\ind{\Bl_{\d}}\right)=\left(k_{c/n}*\left(f_{n}e^{nF_{\!g}}\ind{\Bl_{\d}}\right)\right)\times\ind{\Bl_{2\d}}\,.\]
Hence
\begin{align*}
\left|A_{n,c,5}\right| &\leq\! \int_{\left(\Bl_{\a}\right)^{c}\cap \Bl_{2\d}}\!\!\left(\int_{\R^{2}}\! \left|k_{c/n}(s-x,t-y)\left(f_{n}e^{nF_{\!g}}\ind{\Bl_{\d}}\right)\!(x,y)\right|\,dx\,dy\!\right)d\widetilde{\nu}_{n,\r}(s,t)\\
&\leq \int_{\left(\Bl_{\a}\right)^{c}\cap \Bl_{2\d}} \left( k_{c/n}* e^{nF_{\!g}}\right)(s,t)\,d\widetilde{\nu}_{n,\r}(s,t).
\end{align*}
We note that, for $\d$ small enough, we have on $\Bl_{2\d}$,
\[\left|k_{c/n}* e^{nF_{\!g}}\right|\leq e^{nF_{\!g}} + \left|e^{nF_{\!g}}- k_{c/n}* e^{nF_{\!g}}\right|\leq e^{nF}\left(1+\,2 M \sqrt{2} c\right)\,,\]
if $c/n$ is small enough (we use here the same argument as in the control of $A_{n,c,3}$, with $u=0$). Finally
\[A_{n,c,5}\underset{\gfrac{n\to \infty}{c\to 0}}{=}O\left(\int_{\left(\Bl_{\a}\right)^{c}}e^{nF(s,t)}\,d\widetilde{\nu}_{n,\r}(s,t)\right)\,.\]
This ends the proof of the lemma. \qed

\nocite{Feller2}
\bibliographystyle{plain}
\bibliography{biblio}
\addcontentsline{toc}{section}{References}
\markboth{\uppercase{References}}{\uppercase{References}}

\end{document}